\documentclass[a4paper,12pt]{article}
\usepackage{times}
\usepackage{amsmath,amssymb,amsthm,amsfonts,bm,cite,color,xcolor,float,soul,cite,color,graphicx}

\usepackage[none]{hyphenat}
\usepackage{mathdots}

\usepackage{tikz}
\usepackage{tikz-cd}

\providecommand{\udots}{\Udots}
\makeatletter
\DeclareRobustCommand{\Udots}{
  \vcenter{\offinterlineskip
    \halign{
      \hbox to .75em{##}\cr
      \hfil.\cr\noalign{\kern.15ex}
      \hfil.\hfil\cr\noalign{\kern.15ex}
      .\hfil\cr\noalign{\kern.25ex}}
  }
}

 \usepackage{enumitem}

\numberwithin{equation}{section}

\newtheorem{theorem}{Theorem}[section]
\newtheorem{proposition}[theorem]{Proposition}

\theoremstyle{definition}

\newtheorem{definition}[theorem]{Definition}

\usepackage[center]{caption}
\usepackage{xcolor}
\definecolor{ForestGreen}{RGB}{34,139,34}

\usepackage{hyperref} 
\usepackage{orcidlink}

\newcommand{\slog}{\mathop{\mathrm{slog}}}
\newcommand{\APS}{\mathop{\mathrm{APS}}}

\begin{document}

\begin{center}
{\LARGE \bf Graham's number stable digits: An exact solution}
\vspace{10mm}

{\Large \bf Marco Ripà $^{\orcidlink{0000-0002-6036-5541}}$}
\vspace{6mm}

\end{center}

\noindent {\bf Abstract.} In the decimal numeral system, we prove that the well-known Graham's number, $G := \! ^{n}3$ (i.e., $3^{3^{\udots^{3}}}$  ($n$ times)), and any base $3$ tetration whose hyperexponent is larger than $n$ share the same $\slog_3(G) - 1$ rightmost digits (where $\slog$ indicates the integer super-logarithm). This is an exact result since the $\slog_3(G)$-th rightmost digit of $G$ differs from the $\slog_3(G)$-th rightmost digit of $^{n+1}3$.
Furthermore, we show that the $\slog_3(^{n}3)$-th least significant digit of the difference between Graham's number and any base $3$ tetration whose integer hyperexponent exceeds $n$ is $4$. \\
\vspace{6mm}

\section{Introduction}

Graham's number, $G := g_{64}$, is a large number introduced by Ronald Graham and communicated to Martin Gardner in the late 1970s~\cite{gardner-1977}. In that context, for the sake of simplicity, Graham replaced the actual upper bound of a very hard open problem in the field of Ramsey theory, proved in a 1971 paper written by himself and Bruce Lee Rothschild \cite{98}, with an easier one (which was subsequently cited by the 1980 edition of the ``Guinness Book of World Records'' \cite{99}).

In detail, assuming that each digit of Graham's number occupies one Planck length cubed (i.e., a Planck volume), the observable universe is far too small to contain an ordinary digital representation of $G$. Nevertheless, we can iteratively define Graham's number in a few lines \cite{10}, providing a finite upper bound to the problem of finding the smallest dimension of a hypercube such that every two-colouring of its edges contains a monochromatic complete subgraph of four vertices forming a square~\cite{exoo-2003}.

By construction, Graham's number is an iterated base $3$ hexation (where integer hexation is an iterated pentation, which, in turn, is an iterated tetration, which, in turn, is an iterated exponentiation, and so forth).

In terms of Knuth's up-arrow notation (where 
$a\,\uparrow\,b$ means $a^b$, $a\,\uparrow\,\uparrow\,b$ denotes a power tower of $b$ copies of $a$, and each additional arrow indicates one higher level of iteration \cite{knuth-up-arrows}), we have that
\begin{equation}\label{eq1}
g_k:=\left\{\begin{array}{lll}
3 \uparrow \uparrow \uparrow \uparrow 3 \hspace{5.5mm} \textnormal{if} \hspace{3mm} k = 1 \\
3 \uparrow^{g_{k-1}} 3 \hspace{5mm} \textnormal{if} \hspace{3mm} k \geq 2
\end{array}\right.,
\end{equation}
and then we simply assign $k := 64$ to finally state that $G$ equals $g_{64}$ \cite{10, guy-selfridge-1973}.

By virtue of the fact that $g_{64} = 3 \uparrow \uparrow \uparrow \uparrow (3 \uparrow \uparrow \uparrow \uparrow ( \cdots \uparrow \uparrow \uparrow \uparrow 3)) = G$, we observe that Graham's number is just a very, very high base $3$ tetration, and this lets us trivially conclude that the most significant digit of $G$ is $1$ if we select the binary system, and the same is true if we are assuming the ternary system (since $G$ is an iterated power of $3$ by definition).

On the other hand, the problem of expressing the most significant digit of $G$ in the decimal number system is still open.

Accordingly, the present paper assumes the usual decimal number system, aiming to calculate the exact number of the rightmost (trailing) digits of $G$ that match the tail-end figures of the endless string
$\, \ldots 3967905496638003222348723967018485186439059104575627262464195387$
of frozen digits originated by the increasing height power tower $\underbrace{3^{3^{\udots{^3}}}}_{\text{$b$ times}}$ as $b$ keeps growing~\cite{94}.
\vspace{2mm}

For clarity, we will denote by $\mathbb{N}_0$ the set of nonnegative integers (including zero) and by $\mathbb{N}$ we will denote the set of positive integers $\{1,2,3,\ldots\}$. From here on, let us indicate the integer tetration as in the following Definition~\ref{def1}.

\begin{definition} \label{def1}
We define integer tetration $$a \uparrow \uparrow b := \begin{cases} a, & \textnormal{if~} b = 1  \\
a^{(^{b-1}a)},  & \textnormal{if~} b\geq 2
\end{cases}$$ as $^{b}a$ (where $a,b \in \mathbb{N}$) so that we can compactly write $^{b}3$ instead of $3 \uparrow \uparrow b$ (see also \cite{goebel-nederpelt-1971}), and consequently note that there exists a unique $n \in \mathbb{N}$ such that $^{n}3=G$.
\end{definition}

Thus, the goal of Section 2 is to determine the exact value of the positive integer $r$ such that $^{n+1}3 \equiv G \pmod{10^{r-1}}$ and $^{n+1}3 \not\equiv G \pmod{10^r}$, which is equivalent to asking to find, for any $c \in \mathbb{N}$, the value of $r$ such that $^{n+c}3 \equiv G \pmod{10^{r-1}}$ and $^{n+c}3 \not\equiv G \pmod{10^r}$ \cite{urroz-yebra-2009}.

In the following sections, we will denote by $\slog_3(G)$ the base $3$ super-logarithm of Graham's number (see Definition~\ref{def2}, and see also \cite[p.~68]{50} for the \textit{log star} notation of the \textit{iterated logarithm}) to prove that $r = \slog_3(G)$.

\sloppy Lastly, Section 3 shows that, for each positive integer $c$, the difference between the $\slog_3(G)$-th rightmost digit of $G$ and the corresponding digit of $^{\slog_3(G)+c}3$ is $6$ if the $\slog_3(G)$-th least significant digit of $G$ belongs to the set $\{6,7,8,9\}$, whereas the mentioned difference is equal to $-4$ if this does not occur.

In the end, for every $c \in \mathbb{N}$, $^{\slog_3(G)+c}3 - G \equiv 4 \cdot {10^{\slog_3(G)-1}} \pmod {10^{\slog_3(G)}}$, and so the $\slog_3(G)$-th rightmost digit of $3^G - G$ is $4$.

\section{The last \texorpdfstring{$\slog_3(G)-1$}{slog{3}(G)-1} digits of \texorpdfstring{$G$}{G}}

This section is devoted to proving that the rightmost digit of $G$ that is not equal to the corresponding digit of any base $3$ tetration whose (integer) hyperexponent is greater than $n$ is the $\slog_3(G)$-th, where the mentioned super-logarithm is defined as follows.

\begin{definition} \label{def2}
For positive integers $p$ and $q$, we define the integer super-logarithm $\slog_p(q)$ by $\slog_p(1)=0$ and $\slog_p(p^q) = \slog_p(q)+1$, and, in particular, whenever $\log_p(q)\in\mathbb{N}$ we also have $\slog_p(p^q) = \slog_p(\log_p(q))+2$.
\end{definition}

Consequently, the base $3$ super-logarithm of any positive integer $n$ is such that $\slog_3(^{n}3) = n$, since $\slog_p(1) = 0$ by Definition~\ref{def2}. Hence, $^{n}3 = G$ implies $n = \slog_3(G)$.

Although there is no general agreement on the definition of the super-logarithm for non-integers \cite{33}, this issue does not affect Definition~\ref{def2}, where we restrict to strictly positive integer arguments.

Now, by assuming the decimal number system as usual, we introduce the ``constancy of the congruence speed'', a peculiar property of every tetration base not a multiple of $10$ proved by the author in recent years \cite{88} (see also the related OEIS sequences A317905 and A373387).

\begin{definition} \label{def3}
Let $m \in\mathbb{N}_0$ and assume that $a\in\mathbb{N}-\{1\}$ is not a multiple of $10$. Then, given \linebreak  $^{b-1}a\equiv{^{b}a}\pmod {10^{m}}$ and $^{b-1}a \not\equiv{^{b}a}\pmod {10^{m+1}}$, for all $b \in \mathbb{N}$, $V(a,b)$ returns the nonnegative integer such that $^{b}a\equiv{^{b+1}a}\pmod {10^{m+V(a,b)}}$ and $^{b}a \not\equiv{^{b+1}a}\pmod {10^{m+V(a,b)+1}}$, and we define $V(a,b)$ as the \emph{congruence speed} of the base $a$ at the given height of its hyperexponent~$b$.
Furthermore, let $\bar{b} := \min_{b}\left\{b \in \mathbb{N} : V(a, b)=V(a, b+k) \hspace{1mm} \textnormal{for all} \hspace{1mm} k \in \mathbb{N}\right\}$ so that we define as \emph{constant congruence speed} of $a$ the positive integer $V(a) := V(a, \bar{b})$.
\end{definition}

For any $a, b \in \mathbb{N}-\{1\}$, the congruence speed of the tetration $^{b}a$ should always return the exact number of the rightmost digits of the result that freeze by going from $^{b}a$ to $^{b+1}a$ and that were not already stable digits at height $^{b-1}a$, but there is one (disputable) exception to the above-mentioned general correspondence: it is the case $a=5$ with $b=2$, as $^{1}5=5$, $^{2}5=3125$, and $^{3}5 \equiv \hspace{0mm} ^{2}5 \pmod{10^5}$ so that we could argue that the stated definition would imply $V(5,2)=4=V(5,3)$ even if $^{2}5$ is only a $4$-digit number, where the last one has been already frozen at height $1$ (i.e., \linebreak   $^{2}5 \equiv 5 \pmod{10^1}$); for a proof that this is the only possible collision that can occur, for every $a$ and $b$ greater than $1$, see [\url{https://math.stackexchange.com/questions/4863065/}].

Then, given that Equation~(16) of \cite{96} provides a formula to compute the constant congruence speed of every tetration base not a multiple of $10$, let us invoke it in the proof of Theorem~\ref{Theorem 2}.
More specifically, here we are only interested in line 11 of the mentioned Equation~(16), stating that the constant congruence speed of each tetration base, let us call it ${a}^{\star}$, which is congruent to $3$ modulo $20$ and such that $5$ is not equal to the absolute value of the difference between the rightmost digit of ${a}^{\star}$ that is different from the corresponding digit of $\ldots 9\mkern0.085mu 9\mkern0.085mu 8\mkern0.085mu 3\mkern0.085mu 4\mkern0.085mu 0\mkern0.085mu 3\mkern0.085mu 0\mkern0.085mu 9\mkern0.085mu 7\mkern0.085mu 0\mkern0.085mu 8\mkern0.085mu 9\mkern0.085mu 6\mkern0.085mu 5\mkern0.085mu 7\mkern0.085mu 9\mkern0.085mu 4\mkern0.085mu 8\mkern0.085mu 6\mkern0.085mu 6\mkern0.085mu 6\mkern0.085mu 5\mkern0.085mu 7\mkern0.085mu 7\mkern0.085mu 6\mkern0.085mu 1\mkern0.085mu 3\mkern0.085mu 8\mkern0.085mu 0\mkern0.085mu 2\mkern0.085mu 3\mkern0.085mu 5\mkern0.085mu 4\mkern0.085mu 4\mkern0.085mu 3\mkern0.085mu 1\mkern0.085mu 7\mkern0.085mu 6\mkern0.085mu 6\mkern0.085mu 2\mkern0.085mu 6\mkern0.085mu 6\mkern0.085mu 6\mkern0.085mu 8\mkern0.085mu 3\mkern0.085mu 0\mkern0.085mu 3\mkern0.085mu 6\mkern0.085mu 2\mkern0.085mu 9\mkern0.085mu 7\mkern0.085mu 2\mkern0.085mu 1\mkern0.085mu 8\mkern0.085mu 2\mkern0.085mu 8\mkern0.085mu 0\mkern0.085mu 3\mkern0.085mu 6\mkern0.085mu 4\mkern0.085mu 0\mkern0.085mu 4\mkern0.085mu 7\mkern0.085mu 6\mkern0.085mu 5\mkern0.085mu 8\mkern0.085mu 1\mkern0.085mu 9\mkern0.085mu 0\mkern0.085mu 7\mkern0.085mu 9\mkern0.085mu 2\mkern0.085mu 2\mkern0.085mu 9\mkern0.085mu 4\mkern0.085mu 3$ (i.e., $\left\{5^{2^n}\right\}_{\infty}-\left\{2^{5^n}\right\}_{\infty}$) and the mentioned key digit of $\ldots 362972182803640476581907922943$, is given by the $5$-adic valuation of $({a}^{\star})^2+1$ (and then ${a}^{\star} := \ldots 000003$ implies that the constant congruence speed of ${a}^{\star}$ is equal to the maximum number of times that $5$ divides $3^2+1$, so $V(3) = 1$ follows).

\begin{theorem} \label{Theorem 2}
Let $m$ and $n$ be positive integers such that $m > n$. In the decimal number system, $^{m}3 \equiv G\pmod{10^{n-1}}$ and $^{m}3 \not\equiv G \pmod{10^n}$ occurs if and only if $n = \textnormal{slog}_3(G)$.
\end{theorem}

\begin{proof}
Theorem~\ref{Theorem 2} states that the rightmost digit of Graham's number that is non-stable (i.e., a digit that does not match the corresponding figure of the endless string $\ldots 5627262464195387$ of frozen digits returned by ${3^{3^{3^{\udots}}}}$) is the figure in position $\slog_3(G)$, by counting from right to left of~$G$.

To prove the theorem, we show that, in the decimal number system, the congruence speed of base $3$ tetration is $0$ if and only if the hyperexponent equals $1$, and that it becomes constant starting from height $3$ (since a sufficient condition for any base $a>1$ not a multiple of $10$ is $b \geq \tilde{\nu}(a)+2$, see \cite[Definition~2.1]{96}). For $a=3$ this implies $b \geq 3$, with $V(3,b)=V(3)=1$. Finally, by directly checking $b=2$, we conclude that $V(3,b)=1$ holds for all $b \geq 2$.

It is enough to derive that the total number of the rightmost stable digits of $^{n}3$ is $n - 1$.

In detail, from \cite[Definition~2.1]{96}, it follows that $\tilde{\nu}(3) = \nu_5 \left(3^{2}+1\right)$, where $\nu_5(\text{.})$ indicates the $5$-adic valuation of the argument. Page~449 of the mentioned paper gives $b \geq \tilde{\nu}(a) + 2$ as a sufficient (but not necessary) condition on $b:=b(a)$ for the constancy of the congruence speed of any tetration base $a$ whose last digit is $3$ or $7$.
Since $\nu_5 \left(a^{2}+1\right) + 2 = 3$, we only have to calculate the total number of the rightmost digits of $^{3}3$ that do not change as we move to~$^{4}3$. Then, it is not difficult to see that $^{3}3 \equiv 987 \pmod{10^3}$ and $^{4}3 \equiv 387 \pmod{10^3}$ so that $V(3,1) + V(3,2) + V(3,3) = 2$.

Now, Equation~(16), line~11, of \cite{96} implies that the base $3$ tetration is characterized by a unit constant congruence speed since $V(3)=\nu_5(3^2+1)=\nu_5(5 \cdot 2)=1$.

Hence, given the fact that $G =\hspace{0mm} ^{n}3 = g_{64}$ implies $n > 3$ (see Equation~\eqref{eq1}), the total number of stable digits of $G$ is equal to $2 + (n - 3) \cdot V(3) = 2 + (n - 3) \cdot 1 = n - 1$.

Thus, the least significant digit of $^{n}3$ that is non-stable is the one immediately to the left of the $(n-1)$-th rightmost digit of $^{n}3$.
Since $n = \slog_3(^{n}3) = \slog_3(G)$, we have finally shown that Graham's number has exactly $\slog_3(G) - 1$ frozen digits and the proof is complete.
\end{proof}

Lastly, let us note that we can shorten the proof of Theorem~\ref{Theorem 2} by invoking Lemma~1 of \cite{95} (see pages~246--247). Still, here we have opted for a more general approach not preliminarily restricted by the selection of the only tetration base taken into account by Lemma~1 of that 2020 paper.

\section{The asymptotic phase shift of \texorpdfstring{$3$}{3}}

Since Theorem~\ref{Theorem 2} implies that the $\slog_3(G)$-th rightmost digit of $G$ and the $\slog_3(G)$-th rightmost digit of $3^G$ cannot be equal, we can push our investigation even further by asking, for any given positive integer $c$, which is the $\slog_3(G)$-th rightmost figure of the difference between $^{\slog_3(G)}3$ (i.e., Graham's number) and $^{\slog_3(G)+c}3$ (e.g., if $c := 719$, then we are interested in the $\slog_3(G)$-th least significant digit of $G - ^{\slog_3(G)+719}3$\,).

Given the fact that the $\slog_3(G)$-th rightmost digit of $^{\slog_3(G)}3 - \hspace{0mm}^{\slog_3(G)+c}3$ does not depend on~$c$ (since we are assuming $c \in \mathbb{N}$ by hypothesis) and by observing that $^{\slog_3(G)}3 - \hspace{0mm}^{\slog_3(G)+c}3 < 0$, we will solve the problem above by showing that
\begin{equation}\label{eq2}
\frac{^{\slog_3(G)}3 - \hspace{0mm} ^{\slog_3(G)+1}3}{10^{\slog_3(G)-1}} \equiv 6 \hspace{-3mm} \pmod{10}.
\end{equation}

\begin{proposition}\label{Proposition 1}
Let $n$ and $m$ be two positive integers such that $m > n$. The congruence class modulo $10$ of the difference between the $n$-th rightmost digit of \hspace{1mm}$^{n}3$ and the $n$-th rightmost digit of \hspace{1mm}$^{m}3$ is $4$ or $6$. Furthermore, if $n$ is odd, then the $n$-th least significant digit of $\hspace{1mm}^{n}3 - \hspace{0mm}^{m}3$ is $4$, otherwise the $n$-th least significant digit of $\hspace{1mm}^{n}3 - \hspace{0mm}^{m}3$ is $6$.
\end{proposition}

\begin{definition} \label{def4}
Let $(a,b)$ be a pair of positive integers and consider the decimal number system. For any given tetration base $a \neq 1$ not a multiple of $10$, we call \emph{phase shift} of $a$ at height $b$ the congruence class modulo $10$ of the difference between the rightmost non-stable digit of \hspace{0.5mm}$^{b}a$, say the $(\#S(a,b)+1)$-th by counting positions from right to left (see \cite[Equation~(1),~p.~442]{96}), and the $(\#S(a,b)+1)$-th rightmost digit of $^{b+1}a$.
\end{definition} 

For example, the phase shift of $5$ at height $4$ is $5$ since  $^{4}5 \equiv 68408203125 \pmod{10^{11}}$ and $^{5}5 \equiv 18408203125 \pmod{10^{11}}$ so that \footnotesize $\dfrac{68408203125 - 18408203125}{10^{11-1}}$ \normalsize $ \equiv 5 \pmod{10}$). 

In particular, if $\bar{b}$ (see Definition~\ref{def3}) indicates the smallest hyperexponent of $a$ such that its congruence speed is constant (i.e., $V(a,\bar{b}) = V(a)$), we call \emph{asymptotic phase shift} of $a$ ($\APS(a)$) the $4$-iteration cycle of the phase shifts generated by $\bar{b}$, $\bar{b}+1$, $\bar{b}+2$, and $\bar{b}+3$ (e.g., the asymptotic phase shift of $64$ is the cycle $[8,8,8,8]$ that can be halved twice and thus reduced to $[8]$ since $[8,8,8,8] \rightarrow [8,8] \rightarrow [8]$, while the asymptotic phase shift of $169$ is $[4,8,6,2]$ which cannot be further reduced since $[4,8]$ is different from $[6,2]$).

Let $b \in \mathbb{N}-\{1,2,\ldots,\bar{b}-1\}$. For any given $a$ greater than $1$ and not a multiple of $10$, a noteworthy property of the phase shift is that it only depends on the congruence modulo $4$ of $b$, which means that the phase shift of $a$ at height $b$ is equal to the phase shift of $a$ at height $b+4$, the phase shift of $a$ at height $b+1$ is the same as its phase shift at height $b+5$, and so forth.

In this regard, we note that if the asymptotic phase shift has four distinct elements, then the first entry plus the third one and the second entry plus the fourth one are both equal to $10$, while if the asymptotic phase shift consists of only two entries, their sum is also $10$. 
Moreover, for any given tetration base $a$ as above, all the entries of $\APS(a)$ are constrained to alternatively belong to only one of the following three sets, with no exception (see Appendix): $\{2,4,6,8\}$, $\{1,3,7,9\}$, $\{5\}$.

Consequently, although the asymptotic phase shift can have four distinct entries, it is sufficient to calculate only the phase shifts at heights $b$ and $b+1$ in order to fully map the phase shifts of the given tetration base at every height above $\bar{b} - 1$ (e.g., we can safely set $b := \tilde{\nu}(a)+2$, where $\tilde{\nu}(a)$ is stated by \cite[Definition~2.1,~p.~447]{96}).

In Chapters 3, 4, 6, and 7 of his 2011 monograph \cite{70}, the author analyzed the phase shifts of many integer tetration bases not a multiple of $10$ and the related $4$-iteration cycles (which can sometimes be halved once or twice, depending on $a$).

Here are various examples of the asymptotic phase shift for as many tetration bases that are coprime to $10$:
\begin{center}
\vspace{-0.5mm}
${\APS}(9) = [2]$; \\[1mm]
\vspace{-0.53mm}
${\APS}(11) = [4]$; \\[1mm]
\vspace{-0.53mm}
${\APS}(83) = [6]$; \\[1mm]
\vspace{-0.53mm}
${\APS}(53) = [8]$; \\[1mm]
\vspace{-0.53mm}
${\APS}(33) = [2,8]$, ${\APS}(39) = [8,2]$; \\[1mm]
\vspace{-0.53mm}
${\APS}(43) = [4,6]$, ${\APS}(41) = [6,4]$; \\[1mm]
\vspace{-0.53mm}
${\APS}(7) = [2,6,8,4]$, ${\APS}(73) = [4,2,6,8]$, ${\APS}(31) = [8,4,2,6]$, ${\APS}(29) = [6,8,4,2]$; \\[1mm]
\vspace{-0.53mm}
${\APS}(23) = [2,4,8,6]$, ${\APS}(13) = [6,2,4,8]$, ${\APS}(77) = [8,6,2,4]$, ${\APS}(19) = [4,8,6,2]$; \\[1mm]
\vspace{-0.53mm}
${\APS}(51) = [5]$; \\[1mm]
\vspace{-0.53mm}
${\APS}(101) = [9]$; \\[1mm]
\vspace{-0.53mm}
${\APS}(901) = [1,9]$; \\[1mm]
\vspace{-0.53mm}
${\APS}(301) = [3,9,7,1]$; \\[1mm]
\vspace{-0.53mm}
${\APS}(701) = [7,9,3,1]$.
\end{center}

In particular, if $a = 3$, the asymptotic phase shift is $[4,6,4,6] \rightarrow [4,6]$, and then we alternate between the phase shift $4$ and the phase shift $6$, over and over, according to the parity of $b$.
This modular recurrence is shown in Figure~\ref{fig:Figure1}, where the difference between each red digit and the corresponding dark digit, right below the red one, is $4$ or $-6$ so that $4 \equiv -6 \pmod{10}$ confirms the stated rule.

\begin{figure}[H]
\centering
\includegraphics[width=13cm]{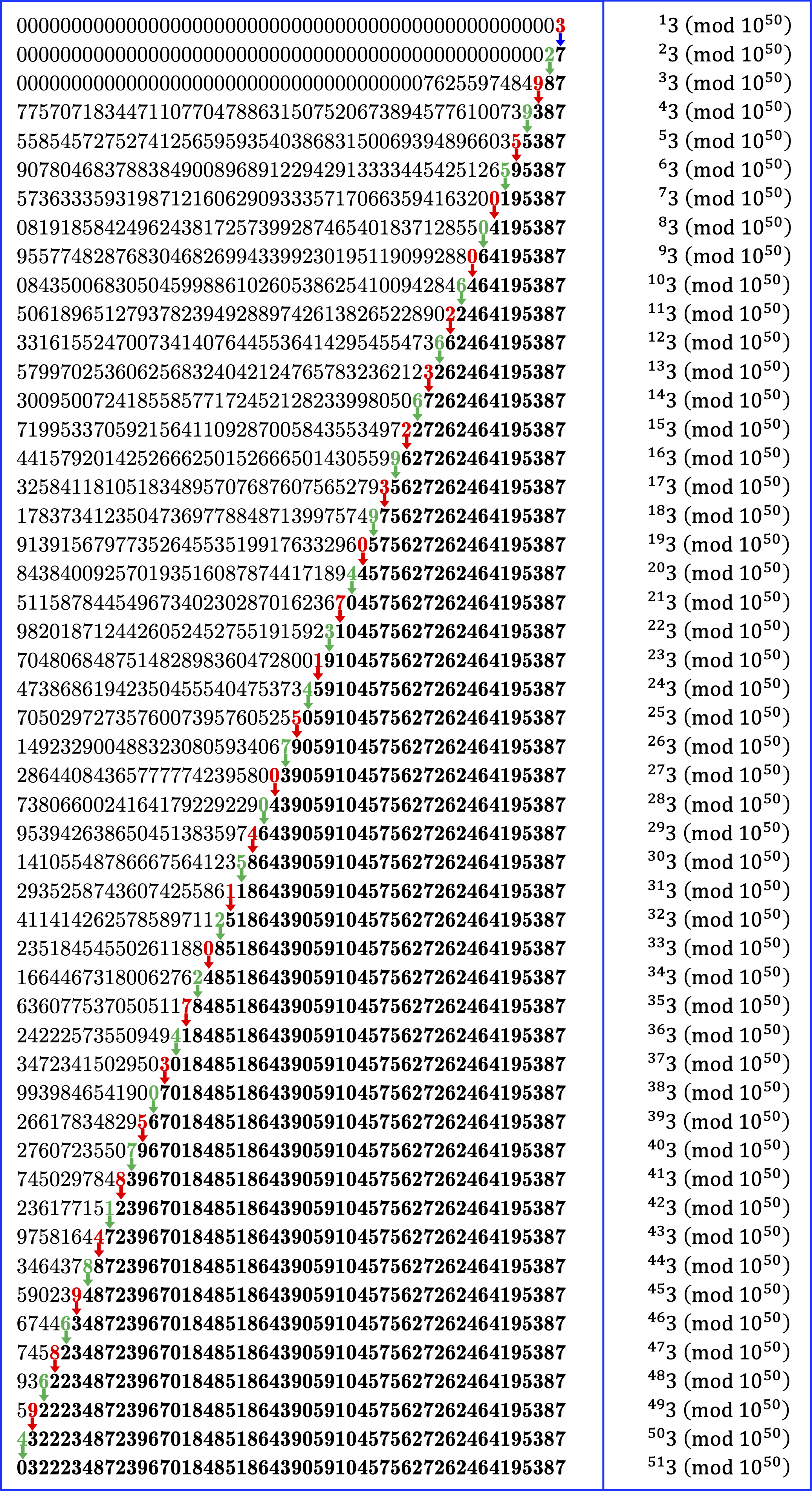}
\caption{Phase shifts of the tetration base $3$ at heights $1, 2, \ldots, 50$.\label{fig:Figure1}}

\end{figure}

Accordingly, we see that the phase shifts of the tetration base $3$ at heights $1$ to $50$ are fully described by the two congruence classes modulo $10$ of the following collection of pairs of differences: $({\color{red} {3-7}}$, $ {\color{ForestGreen}{2-8}}$; ${\color{red} {9-3}}$, $ {\color{ForestGreen} {9-5}}$; ${\color{red} {5-9}}$, $ {\color{ForestGreen} {5-1}}$; ${\color{red} {0-4}}$, $ {\color{ForestGreen} {0-6}}$; ${\color{red} {0-4}}$, $ {\color{ForestGreen} {6-2}}$; ${\color{red} {2-6}}$, $ {\color{ForestGreen} {6-2}}$; ${\color{red} {3-7}}$, $ {\color{ForestGreen} {6-2}}$; ${\color{red} {2-6}}$, $ {\color{ForestGreen} {9-5}}$; ${\color{red} {3-7}}$, $ {\color{ForestGreen} {9-5}}$; ${\color{red} {0-4}}$, ${\color{ForestGreen} {4-0}}$; ${\color{red} {7-1}}$, ${\color{ForestGreen} {3-9}}$; ${\color{red} {1-5}}$, ${\color{ForestGreen} {4-0}}$; ${\color{red} {5-9}}$, ${\color{ForestGreen} {7-3}}$; ${\color{red} {0-4}}$, ${\color{ForestGreen} {0-6}}$; ${\color{red} {4-8}}$, ${\color{ForestGreen} {5-1}}$; ${\color{red} {1-5}}$, ${\color{ForestGreen} {2-8}}$; ${\color{red} {0-4}}$, ${\color{ForestGreen} {2-8}}$; ${\color{red} {7-1}}$, ${\color{ForestGreen} {4-0}}$; ${\color{red} {3-7}}$, ${\color{ForestGreen} {0-6}}$; ${\color{red} {5-9}}$, ${\color{ForestGreen} {7-3}}$; ${\color{red} {8-2}}$, ${\color{ForestGreen} {1-7}}$; ${\color{red} {4-8}}$, ${\color{ForestGreen} {8-4}}$; ${\color{red} {9-3}}$, ${\color{ForestGreen} {6-2}}$; ${\color{red} {8-2}}$, ${\color{ForestGreen} {6-2}}$; ${\color{red} {9-3}}$, ${\color{ForestGreen} {4-0}}$; $\ldots)$.

Hence, we have $({\color{red} -4}$, $ {\color{ForestGreen} -6}$; ${\color{red} 6}$, ${\color{ForestGreen} 4}$; ${\color{red} -4}$, ${\color{ForestGreen} 4}$; ${\color{red} -4}$, $ {\color{ForestGreen} -6}$; ${\color{red} -4}$, $ {\color{ForestGreen} 4}$; ${\color{red} -4}$, ${\color{ForestGreen} 4}$; ${\color{red} -4}$, ${\color{ForestGreen} 4}$; ${\color{red} -4}$, ${\color{ForestGreen} 4}$; ${\color{red} -4}$, $ {\color{ForestGreen}4}$; ${\color{red} -4}$, $ {\color{ForestGreen}4}$; ${\color{red} 6}$, ${\color{ForestGreen}-6}$; ${\color{red} -4}$, ${\color{ForestGreen}4}$; ${\color{red} -4}$, $ {\color{ForestGreen}4}$; ${\color{red} -4}$, ${\color{ForestGreen}-6}$; ${\color{red} -4}$, ${\color{ForestGreen}4}$; ${\color{red} -4}$, ${\color{ForestGreen}-6}$; ${\color{red} -4}$, ${\color{ForestGreen}-6}$; ${\color{red} 6}$, ${\color{ForestGreen}4}$; ${\color{red} -4}$, ${\color{ForestGreen}-6}$; ${\color{red} -4}$, ${\color{ForestGreen}4}$; ${\color{red} 6}$, ${\color{ForestGreen}-6}$; ${\color{red} -4}$, ${\color{ForestGreen}4}$; ${\color{red} 6}$, ${\color{ForestGreen}4}$; ${\color{red} 6}$, ${\color{ForestGreen}4}$; ${\color{red} 6}$, ${\color{ForestGreen}4}$; $\ldots)$, which can be split into the odd-hyperexponent difference collection (red entries) $({\color{red} -4}$, ${\color{red} 6}$, ${\color{red} -4}$, ${\color{red} -4}$, ${\color{red} -4}$, ${\color{red} -4}$, ${\color{red} -4}$, ${\color{red} -4}$, ${\color{red} -4}$, $ {\color{red} -4}$, ${\color{red} 6}$, $ {\color{red} -4}$, ${\color{red} -4}$, $ {\color{red} -4}$, ${\color{red} -4}$, $ {\color{red} -4}$, ${\color{red} -4}$, $ {\color{red} 6}$, ${\color{red} -4}$, $ {\color{red} -4}$, ${\color{red} 6}$, $ {\color{red} -4}$, ${\color{red} 6}$, $ {\color{red} 6}$, ${\color{red} 6}$, $ \ldots)$ and the even-hyperexponent difference collection (green entries) $\hspace{1mm}({\color{ForestGreen} -6}$, ${\color{ForestGreen} 4}$, ${\color{ForestGreen} 4}$, ${\color{ForestGreen} -6}$, ${\color{ForestGreen} 4}$, ${\color{ForestGreen} 4}$, ${\color{ForestGreen} 4}$, ${\color{ForestGreen} 4}$, ${\color{ForestGreen} 4}$, ${\color{ForestGreen} 4}$, ${\color{ForestGreen} -6}$, ${\color{ForestGreen} 4}$, ${\color{ForestGreen} 4}$, ${\color{ForestGreen} -6}$, ${\color{ForestGreen} 4}$, ${\color{ForestGreen} -6}$, ${\color{ForestGreen} -6}$, ${\color{ForestGreen} 4}$, ${\color{ForestGreen} -6}$, ${\color{ForestGreen} 4}$, ${\color{ForestGreen} -6}$, ${\color{ForestGreen} 4}$, ${\color{ForestGreen} 4}$, ${\color{ForestGreen} 4}$, ${\color{ForestGreen} 4}$, $\ldots)$.

Thus, each red entry belongs to the congruence class $6$ modulo $10$ while all the green entries are congruent to $4$ modulo $10$.

For every positive integer $b$, we easily obtain the $b$-th rightmost digit of $^{b+1}3$ from the $b$-th rightmost digit of $^{b}3$ by knowing that the asymptotic phase shift of $3$ is $[4,6]$. In detail, for each positive odd integer $b$, $4$ plus the $b$-th least significant digit of \hspace{0.0mm}$^{b}3$ is congruent modulo $10$ to the $b$-th least significant digit of \hspace{0.0mm}$^{b+1}3$ (see Figure~\ref{fig:Figure1}). Conversely, for each positive even integer $b$, $6$~plus the $b$-th least significant digit of \hspace{0.5mm}$^{b}3$ is congruent modulo $10$ to the $b$-th least significant digit of \hspace{0.5mm}$^{b+1}3$.

Let us call $B_0$ the $b$-th rightmost digit of $^{b}3$ and let $B_1$ indicate the $b$-th rightmost digit of $^{b+1}3$ so that
\begin{itemize}[itemsep=1mm,topsep=1mm]
 \item if $b \equiv 1\pmod{2}$ and $B_0 \in \{0,1,2,3,4,5\}$, then $B_1 = B_0 + 4$;
 \item if $b \equiv 1\pmod{2}$ and $B_0 \in \{6,7,8,9\}$, then $B_1 = ((B_0 + 4) - 10)$ and thus $B_1 = B_0 - 6$;
 \item if $b \equiv 0\pmod{2}$ and $B_0 \in \{0,1,2,3\}$, then $B_1 = B_0 + 6$;
 \item if $b \equiv 0\pmod{2}$ and $B_0 \in \{4,5,6,7,8,9\}$, then $B_1 = ((B_0 + 6) - 10)$ and thus $B_1 = B_0 - 4$.
\end{itemize}

Thus, given $a := 3$, if $b \equiv 1 \pmod{2}$, then $0 \mapsto 4$, $1 \mapsto 5$,
 $2 \mapsto 6$, $3 \mapsto 7$, $4 \mapsto 8$, $5 \mapsto 9$, $6 \mapsto 0$, $7 \mapsto 1$, $8 \mapsto 2$, $9 \mapsto 3$, whereas if $b \equiv 0 \pmod{2}$, then $0 \mapsto 6$, $1 \mapsto 7$, $2 \mapsto 8$, $3 \mapsto 9$, $4 \mapsto 0$, $5 \mapsto 1$, $6 \mapsto 2$, $7 \mapsto 3$, $8 \mapsto 4$, $9 \mapsto 5$.

As a consequence, for any given $c \in \mathbb{N}$, the difference between the $\slog_3(G)$-th rightmost digit of Graham's number and the $\slog_3(G)$-th rightmost digit of $^{\slog_3(G)+c}3$ belongs to the set $\{-4, 6\}$, whereas the difference between the ($\slog_3(G)+1$)-th rightmost digit of $^{\slog_3(G)+1}3$ and the ($\slog_3(G)+1$)-th rightmost digit of $^{\slog_3(G)+1+c}3$ is necessarily equal to $-6$ or $4$.

Now, we observe that $\slog_3(G)$ is congruent to $1$ modulo $2$ since $G$ is an iterated base $3$ integer pentation and $3$ is congruent to $1$ modulo $2$, so we are interested only in the elements of the odd-hyperexponent phase shift set (which we have already verified to belong to the congruence class $6$ modulo $10$).

Therefore, for any given positive integer $c$, we can finally conclude that the $\slog_3(G)$-th least significant digit of $G - \hspace{0.0mm} ^{\slog_3(G)+c}3$ is $4$ (while the ($\slog_3(G)+1$)-th least significant digit of \linebreak $3^G -\hspace{0.0mm} ^{\slog_3(G)+1+c}3$ is $6$, indeed).

\section{Conclusion}

From the trivial observation that Graham's number, $G$, satisfies (by definition) $G = \! ^{\slog_3(G)}3$, we have shown that, in the decimal number system, the last $\slog_3(G) - 1$ digits of $G$ are the same as any power tower of the form $3^{3^{\udots{^3}}}$ which is at least equal to $^{\slog_3(G)}3$. Then, as the height of the mentioned power tower goes beyond $\slog_3(G)$, the difference between the $\slog_3(G)$-th rightmost digit of $G$ and the $\slog_3(G)$-th rightmost digit of $3^{3^{\udots{^3}}}$ becomes congruent to $6$ modulo $10$, and the stated congruence class does not change anymore as the height of $3^{3^{\udots{^3}}}$ continues to grow.

In conclusion, for any positive integer $c$, we can state that $4$ is the $\slog_3(G)$-th least significant digit of $^{\slog_3(G)+c}3 - G$ while $^{\slog_3(G)+c}3 - G \equiv 0 \pmod {10^{\slog_3(G)-1}}$.

\makeatletter
\renewcommand{\@biblabel}[1]{[#1]\hfill}
\makeatother

\section*{Appendix} \label{sec:Append}

The asymptotic phase shift is well defined only for tetration bases that are not a multiple of $10$ (Definition~\ref{def4}) since, for every $a$ ending with $0$, there is no hyperexponent $\bar{b} \in \mathbb{N}$ such that the congruence speed of $a$ is constant, and Equation~(2) of \cite{90} describes the very fast growth of the congruence speed of the (nonzero) tetration bases ending with $0$.

Nevertheless, let $10^c \mid a$ and $10^{c+1} \nmid a$, where $c$ is a strictly positive integer (as usual). Then, call $h(a)$ the rightmost nonzero digit of such tetration bases multiple of $10$ (e.g., $h(200050) = 5$ and $h(81743000) = 3$). We note that, for each hyperexponent $b \in \mathbb{N}-\{1,2\}$, the phase shift of $a$ at height $b$ is trivially equal to the least significant nonzero digit of $\left(\frac{a}{10^c}\right)^{(^{b-1}a)}$, and thus it is $1$ if and only if $h(a) \in \{1,3,7,9\}$, $6$ if and only if $h(a) \in \{2,4,6,8\}$, and $5$ if and only if $h(a) = 5$ (since $b > 2$ implies that ${^{b-1}}a$ is a multiple of $100$ and, consequently, is congruent to $0\pmod{4}$ so that, as long as $10^c \mid a$, the phase shift of $a$ at any height $b \geq 3$ equals the rightmost nonzero digit of $\left(\frac{a}{10^c}\right)^4$).

If the last digit of $a$ is $5$ (i.e., $5 \mid a$ and $2 \nmid a$), we observe that the asymptotic phase shift is always $[5]$ and, as we also assume that $a \neq 5$, the congruence speed of the mentioned tetration bases is certainly constant from height $3$ (see \cite[p.~448]{96}).

For every $a \in \mathbb{N}-\{1\}$ not a multiple of $10$, we conjecture that $\APS(a) \in \{[2],$ $[4],$ $[6],$ $[8],$ $[2,8],$ $[8,2],$ $[4,6],$ $[6,4],$ $[2,6,8,4],$ $[4,2,6,8],$ $[8,4,2,6],$ $[6,8,4,2],$ $[2,4,8,6],$ $[6,2,4,8],$ $[8,6,2,4],$ $[4,8,6,2],$ $[5],$ $[9],$ $[1,9],$ $[3,9,7,1],$ $[7,9,3,1] \}$ (and furthermore, if $a$ is also even, only the cases $[2],$ $[4],$ $[6],$ $[8],$ $[2,8],$ $[8,2],$ $[4,6],$ $[6,4],$ $[4,2,6,8],$ $[8,4,2,6],$ $[6,8,4,2],$ $[2,4,8,6],$ $[6,2,4,8],$ and $[4,8,6,2]$ can occur, given the fact that the congruence speed of every tetration base congruent to $2$, $4$, $6$, $8$ modulo $10$ becomes constant at most at height $3$, see \cite[Section~2.1]{96}).

\end{document}